\documentclass[12pt,reqno,sort]{elsarticle}

\usepackage{amsfonts,color,amsmath,amssymb,fancyhdr}
\usepackage{amsthm}

\usepackage{graphicx}

\usepackage{subfigure}

\usepackage{multirow}
\usepackage{mathrsfs}
\usepackage{booktabs}
\usepackage{longtable}
\usepackage[cal=cm]{mathalfa}
\def\Box{\vcenter{\vbox{\hrule\hbox{\vrule
				\vbox to 8.8pt{\hbox to 10pt{}\vfill}\vrule}\hrule}}}

\newtheorem{thm}{Theorem}[section]
\newtheorem{lemma}[thm]{Lemma}
\newtheorem{corollary}[thm]{Corollary}
\newtheorem{prop}[thm]{Proposition}

\numberwithin{equation}{section}

\newtheorem{hypothesis}[thm]{Hypothesis}

\newcommand{\cP}{\mathcal P}

\definecolor{Purple}{rgb}{0.5,0,0.5}

\def\a{\alpha}

\newcommand{\Gg}{\Gamma}

\begin{document}
	\newcommand{\stopthm}{\begin{flushright}
			\(\box \;\;\;\;\;\;\;\;\;\; \)
	\end{flushright}}
	\newcommand{\symfont}{\fam \mathfam}
	
	\title{Flag-transitive point-primitive quasi-symmetric $2$-designs with block intersection numbers $0$ and $y\leq10$}
	
	\date{}
	
	\author[add1]{Jianbing Lu\corref{cor1}}\ead{jianbinglu@nudt.edu.cn}\cortext[cor1]{Corresponding author}
	\author[add2]{Yu Zhuang}\ead{22135044@zju.edu.cn}
	\address[add1]{Department of Mathematics, National University of Defense Technology, Changsha 410073, China}
	\address[add2]{School of Mathematical Sciences, Zhejiang University, Hangzhou 310027, Zhejiang, P.R.  China}
	
	\begin{abstract}
		In this paper, we show that for a non-trivial quasi-symmetric $2$-design $\mathcal{D}$ with two block intersection numbers $x=0$ and $2\leq y\leq10$, if $G\leq \mathrm{Aut}(\mathcal{D})$ is flag-transitive and point-primitive, then $G$ is either of  affine type or almost simple type. Moreover,  we prove that the socle of $G$ cannot be an  alternating group. If the socle of $G$ is a sporadic group, then $\mathcal{D}$ and $G$ must be one of the following: $\mathcal{D}$ is a $2$-$(12,6,5)$ design with block intersection numbers $0,3$  and $G=\mathrm{M}_{11}$, or $\mathcal{D}$ is a $2$-$(22,6,5)$ design with block intersection numbers  $0,2$ and $G=\mathrm{M}_{22}$ or $\mathrm{M}_{22}:2$.
		\newline
		
		\noindent\text{Keywords:} quasi-symmetric $2$-design; flag-transitive; point-primitive; automorphism group
		
	\end{abstract}	
	
	\maketitle
	
	\section{Introduction}\label{introduction}
	
	A $2$-$(v,k,\lambda)$ design $\mathcal{D}=(\mathcal{P},\mathcal{B})$ is  a finite incidence structure  with a set $\mathcal{P}$ of $v$ points and a set $\mathcal{B}$ of blocks such that each block contains $k$ points and each two points are contained in  $\lambda$ blocks. The \emph{replication number} $r$ of $\mathcal{D}$ is the number of blocks containing a given point. The number of blocks is conventionally denoted $b$. If $b=v$, we say that $\mathcal{D}$ is \emph{symmetric}. It is \emph{non-trivial} if $2<k<v-1$. All the $2$-designs in this paper are assumed to be non-trivial. An \emph{automorphism} of $\mathcal{D}$ is a permutation of the points which preserves the blocks. We write $\mathrm
	{Aut}(\mathcal{D})$ for the full automorphism group of $\mathcal{D}$, and call its subgroups as automorphism groups. A \emph{flag} of  $\mathcal{D}$ is an incident point-block pair. We say that an automorphism group $G$ is \emph{flag-transitive} if it acts transitively on the flags of $\mathcal{D}$.   The group $G$ is said to be  \emph{point-primitive} if $G$ acts  primitively  on  $\mathcal{P}$. 
	
	Suppose  $\mathcal{D}$ is a $2$-$(v,k,\lambda)$ design with  blocks $B_1,B_2,\cdots,B_b$. The cardinality $|B_i\cap B_j|$ is called a \emph{block intersection number} of $\mathcal{D}$. A symmetric $2$-design  has only one block intersection number,   namely $\lambda$.  As a slight generalization of symmetric $2$-designs,  $2$-designs with   two block intersection numbers  are called \emph{quasi-symmetric}. This concept  goes back to \cite{Shrikhande1952}. Let $x$, $y$ denote the two block intersection numbers of a quasi-symmetric design  with the standard convention that $x\leq y$. We say that quasi-symmetric design is proper if $x<y$ and improper if $x=y$. Any linear space with $b>v$ is a quasi-symmetric design with $x=0$ and $y=1$. We refer to \cite{Shrikhande} for more details about quasi-symmetric designs. 
	
	The problem of classifying quasi-symmetric $2$-designs appears to  be very
	difficult, even for the case $x=0$, cf. \cite{Neumaier1982,McDonough1995,Mavron1989}. 
	  However, there have been extensive works on the classification of linear spaces and symmetric $2$-designs, which are close to the quasi-symmetric $2$-designs. Suppose that $\mathcal{L}$ is  a linear space with a flag-transitive automorphism group $G$. According to \cite{Higman1969} $G$ is point-primitive, and by \cite{Buekenhout1988} $G$ is either almost simple or of affine type. Finally, all flag-transitive linear spaces apart from those with a one-dimensional affine automorphism group were classified in  \cite{Buekenhout1990}. Through a series of papers \cite{O'Reilly2005,O'Reilly2005_2,O'Reilly2007,O'Reilly2008},  Regueiro gave the similar classification  of symmetric biplanes ($\lambda=2$). A flag-transitive and point-primitive automorphism group $G$ of a symmetric design with $\lambda\leq100$ must be of almost simple or of affine type \cite{Tian2013}.  In the case where $G$ is imprimitive,  all symmetric designs for $\lambda\leq10$ are determined in \cite{Praeger2006,Mandi}. More recent and interesting classification results are provided in \cite{Li2024, Montinaro2023, Montinaro2024}.  There are few works on flag-transitive  quasi-symmetric $2$-designs with $x=0$ and $y>1$. The only result in the current study known to the authors, is that of Zhang and Zhou in \cite{Zhang2023}. They proved that if $\mathcal{D}$ is a quasi-symmetric design with $x=0$ and $y=2$ satisfying that $\lambda$ does not divide $r$ and $G \leq \mathrm{Aut}(\mathcal{D})$ is flag-transitive, then $G$ is point-primitive. Furthermore,  $G$ is of affine or almost simple type. Based on this conclusion, they also showed that  if $G$ is an almost simple group with sporadic socle, then  $\mathcal{D}$ is the unique $2$-$(22,6,5)$ design  and $G=\mathrm{M}_{22}$ or $\mathrm{M}_{22}:2$. The purpose of the present paper is to generalize the results in \cite{Zhang2023} by using similar methods. The following are our main results.
	
	\begin{thm}\label{main}
		Let $\mathcal{D}$ be a non-trivial quasi-symmetric $2$-design with block intersection numbers $x=0$ and $y\leq10$. If $G\leq\mathrm{Aut}(\mathcal{D})$ is flag-transitive and point-primitive, then $G$ must be either of affine type or almost simple type.
	\end{thm}
	
	\begin{thm}\label{second}
		Let $\mathcal{D}$ be a non-trivial quasi-symmetric $2$-design with block intersection numbers $x=0$ and $2\leq y\leq 10$. If $G\leq\mathrm{Aut}(\mathcal{D})$ is flag-transitive and point-primitive, then the socle of $G$ cannot be an alternating group.
	\end{thm}
	
	In \cite{Delandtsheer2001}, Delandtsheer classified the  linear spaces admitting  flag-transitive almost simple groups with alternating socle.   Combining Theorem \ref{second} with \cite{Delandtsheer2001}, one obtains the following Corollary. 
	
	\begin{corollary}\label{second1}
		Let $\mathcal{D}$ be a non-trivial quasi-symmetric $2$-design with block intersection numbers $x=0$ and $y\leq 10$. If $G\leq\mathrm{Aut}(\mathcal{D})$ is flag-transitive and point-primitive with alternating socle, then   $\mathcal{D}$ is the $2$-$(15,7,1)$ design $\mathrm{PG}(3,2)$ and $G=\mathrm{A}_7$ or $\mathrm{A}_8$.
	\end{corollary}
	
	\begin{thm}\label{third}
		Let $\mathcal{D}$ be a non-trivial quasi-symmetric $2$-design with block intersection numbers $x=0$ and $y\leq 10$. If $G\leq\mathrm{Aut}(\mathcal{D})$ is flag-transitive and point-primitive with sporadic socle, then $\mathcal{D}$ and $G$ are one of the following:
		\begin{enumerate}
			\item[(1)] $\mathcal{D}$ is the unique $2$-$(12,6,5)$ design with block intersection numbers $0$ and $3$, $G=\mathrm{M}_{11}$;
			\item[(2)] $\mathcal{D}$ is the unique $2$-$(22,6,5)$ design with block intersection numbers $0$ and $2$, $G=\mathrm{M}_{22}$ or $\mathrm{M}_{22}:2$.
		\end{enumerate}
	\end{thm}
	
	This paper is organized as follows. In Section \ref{s2}, we present some preliminary results on flag-transitive $2$-designs and the maximal subgroups of almost simple groups with alternating socle.  The proof of Theorem \ref{main} is presented in Section \ref{s3},  and Theorem \ref{second} is detailed in Section \ref{s4}. Finally, the proof of Theorem \ref{third} is given in Section \ref{s5}.
	
	\section{Preliminaries}\label{s2}
	
	\subsection{Flag-transitive $2$-designs}
	We first present some preliminary results about flag-transitive $2$-designs which are used throughout the paper.
	
	\begin{lemma}\label{s21}
		\cite[1.2 and 1.9]{Colburn} Let $\mathcal{D}$ be a $2$-design with the parameters $(v,b,r,k,\lambda)$. Then the following hold:
		\begin{enumerate}
			\item[(i)] $r(k-1)=\lambda(v-1)$, $r>\lambda$;
			\item[(ii)] $vr=bk$;
			\item[(iii)] $b\geq v$, $k\leq r$. If $\mathcal{D}$ is non-symmetric, then $b>v$ and $k<r$.
		\end{enumerate}
	\end{lemma}
	
	\begin{lemma}\label{s23}
		\cite[Lemma 2.4]{Zhang2023}	Let $\mathcal{D}$ be a $2$-design with the parameters $(v,b,r,k,\lambda)$, $\alpha$ be a point of $\mathcal{D}$ and $G$ be a flag-transitive automorphism group of $\mathcal{D}$. Then the following hold:
		\begin{enumerate}
			\item[(i)] %$r^{2}\geq rk>\lambda v$ and
			$|G_{\alpha}|^{3}>\lambda|G|$;
			\item[(ii)] $r\mid\lambda(v-1,|G_{\alpha}|)$, where $G_{\alpha}$ is the stabilizer of $\alpha$ in $G$;
			\item[(iii)] If $d$ is any non-trivial subdegree of $G$, then $r\mid\lambda d$. Moreover, $\frac{r}{(r,\lambda)}\mid d$.
		\end{enumerate}
	\end{lemma}
	
	It is well known that $\mathcal{D}$ is a quasi-symmetric design with intersection numbers $0$ and $1$ if  and only if $\mathcal{D}$ is a finite linear space with  $b>v$ (see \cite[Chapter 3]{Shrikhande}). For the linear space  $\mathcal{D}$ with a flag-transitive automorphism group $G$, we know that $G$ is point-primitive. Moreover, $G$ is either almost simple or of affine type. The cases where the socle of $G$ is an
	alternating group or a sporadic group  were handled  in \cite{Delandtsheer2001,Buekenhoutunpublished,Davies1987phd}. In particular, there are no non-trivial examples if the socle of $G$ is a sporadic group. Thus we suppose that $\mathcal{D}$ is a quasi-symmetric $2$-design with block intersection numbers $x=0$ and $y\geq2$. Then we have the following Lemma.

	\begin{lemma}\label{s22}
		Let $\mathcal{D}$ be a non-trivial quasi-symmetric $2$-design with block intersection numbers $x=0$ and $y\geq2$. Then the following relations hold:
		\begin{enumerate}
			\item[(i)] $(y-1)(r-1)=(k-1)(\lambda-1)$;
			\item[(ii)] $b\leq v(v-1)/k$, $y<\lambda\leq k-1$;
			\item[(iii)] $y\mid k$, $y\mid(r-\lambda)$;
			\item[(iv)] $v<\frac{k^{2}-k}{y-1}$ and $v\leq2(y-1)\cdot\frac{r^{2}}{(r,\lambda)^{2}}$. 
		\end{enumerate}
	\end{lemma}
	
	\begin{proof}
		For part (i) and (ii), it suffices to  prove that $\lambda\leq k-1$ since other  properties are taken from \cite[Proposition 3.17 (i)(ii)]{Shrikhande}. Note that $\lambda=r(k-1)/(v-1)$  and $b=vr/k$ by Lemma \ref{s21}(i)(ii). Thus we have $r\leq v-1$ and so $k-1\geq\lambda$ by (ii).
		Part (iii) is taken from \cite[Corollary 3.9]{Shrikhande}. For more details, see \cite[Theorem 3.8]{Shrikhande}. By combining with (i) and   Lemma \ref{s21}(i),  we have $\frac{v-1}{k-1}=\frac{r}{\lambda}<\frac{r-1}{\lambda-1}=\frac{k-1}{y-1}$. Then $v< \frac{(k-1)^2}{y-1}+1=\frac{k^{2}-2k+y}{y-1}< \frac{k^{2}-k}{y-1}$ since $y<k$. On the other hand, we have $\frac{k-1}{y-1}=\frac{r-1}{\lambda-1}<\frac{2r}{\lambda}$ since $\lambda \geq2$. Hence $v-1=\frac{r(k-1)}{\lambda}<2(y-1)\cdot\frac{r^{2}}{\lambda^{2}}\leq2(y-1)\cdot\frac{r^{2}}{(r,\lambda)^{2}}$ and so $v\leq2(y-1)\cdot\frac{r^{2}}{(r,\lambda)^{2}}$, as stated in part (iv).
	\end{proof}

	\subsection{Maximal subgroups of the almost simple groups with alternating socle}
	We need the following results about maximal subgroups of the almost simple groups with alternating socle to prove Theorem \ref{second}.
	
	\begin{lemma}\label{max}
		\cite[p.366]{Liebeck1987} If $G$ is $\mathrm{A}_{n}$ or $\mathrm{S}_{n}$, acting on a set $\Omega$ of size $n$, and $H$ is any maximal subgroup of $G$ with $H\neq \mathrm{A}_{n}$, then $H$ satisfies one of the following:
		\begin{enumerate}
			\item[(i)] $H=(\mathrm{S}_{l}\times\mathrm{S}_{m})\cap G$, with $n=l+m$ and $l\neq m$ (intransitive case);
			\item[(ii)] $H=(\mathrm{S}_{l}\wr\mathrm{S}_{m})\cap G$, with $n=lm$, $l>1$, $m>1$ (imprimitive case);
			\item[(iii)] $H=\mathrm{AGL}(m,p)\cap G$, with $n=p^{m}$ and $p$ prime (affine case);
			\item[(iv)] $H=(T^{m}.(\mathrm{Out}(T)\times\mathrm{S}_{m}))\cap G$, with $T$ a non-abelian simple group, $m\geq2$ and $n=|T|^{m-1}$ (diagonal case);
			\item[(v)] $H=(\mathrm{S}_{l}\wr\mathrm{S}_{m})\cap G$, with $n=l^{m}$, $l\geq5$ and $m>1$ (wreath case);
			\item[(vi)] $T\trianglelefteq H\leq\mathrm{Aut}(T)$, with $T$ a non-abelian simple group, $T\neq \mathrm{A}_{n}$ and $H$ acting primitively on $\Omega$ (almost simple case).
		\end{enumerate}
	\end{lemma}
	
	\begin{lemma}\label{odd}
		\cite[Theorem (b)(I)]{Liebeck} Let $G$ be a primitive permutation group of odd degree $n$ on a set $\Omega$ and let $H=G_\alpha$, where $\alpha$ in $\Omega$.  If $\mathrm{Soc}(G)$ is $\mathrm{A}_c$, an alternating group, then one of the following holds:
		\begin{enumerate}
			\item[(i)]  $H$ is intransitive, and $H=\left(\mathrm{S}_a\times \mathrm{S}_{c-a}\right)\cap G$ where $1\leq a<\frac{1}{2} c$;
			\item[(ii)]  $H$ is transitive and imprimitive, and $H=\left(\mathrm{S}_a\wr \mathrm{S}_{c/a}\right)\cap G$ where $a>1$ and $a \mid c$;
			\item[(iii)]  $H$ is primitive, $n=15$ and $G$ is $\mathrm{A}_7$ or  $\mathrm{A}_8$.
		\end{enumerate}
	\end{lemma}
	
	\begin{lemma}\label{large}
		
		\cite[Lemma 2.6]{Zhu2016} Suppose that $G=\mathrm{A}_{n} $ or $\mathrm{S}_{n}$, where $n\geq 5$ and $H$ is  a maximal primitive subgroup of $G$ such that $|H|^{3}>|G|$.
		
		\begin{enumerate}
			\item[(i)] \cite[Theorem 2]{Alavi} If $G=\mathrm{A}_{n}$, then $(n, H)$ is one of the following:
			\begin{equation*}
				\begin{aligned} 
					&(5,\mathrm{D}_{10}), &&(6,\mathrm{PSL}(2,5)), &&(7,\mathrm{PSL}(2,7)), &&(8,\mathrm{AGL}(3,2)),\\
					&(9,3^{2}.\mathrm{SL}(2,3)), &&(9,\mathrm{P}\Gamma\mathrm{L}(2,8)), &&(10,\mathrm{M}_{10}),&&(11,\mathrm{M}_{11}),\\
					&(12,\mathrm{M}_{12}), &&(13,\mathrm{PSL}(3,3)), &&(15,\mathrm{A}_{8}), &&(16,\mathrm{AGL}(4,2)),\\
					&(24,\mathrm{M}_{24}).& && &&& \\
				\end{aligned}\\[6pt]
				\nonumber
			\end{equation*}		
			\item[(i)] If $G=\mathrm{S}_{n}$, then $H=\mathrm{A}_{n},\mathrm{S}_{n-1}$ or $(n,H)$ is one of the following:
			\begin{equation*}
				\begin{aligned} 
					&(5,\mathrm{AGL}(1,5)), &&(6,\mathrm{PGL}(2,5)), &&(7,\mathrm{AGL}(1,7)), &&(8,\mathrm{PGL}(2,7)),\\
					&(9,\mathrm{AGL}(2,3)), &&(10,\mathrm{P}\Gamma\mathrm{L}(2,9)), &&(12,\mathrm{PGL}(2,11)).&&\\
				\end{aligned}\\[6pt]
				\nonumber
			\end{equation*}
		\end{enumerate}
	\end{lemma}
	
	\section{Proof of Theorem \ref{main}}\label{s3}

	Firstly, we recall the O'Nan Scott Theorem \cite{Liebeck1988}. Let $M=\mathrm{Soc}(G)$,  where $\mathrm{Soc}(G)$ is the product of all minimal normal subgroups of $G$.
	If $G$ is a finite primitive permutation group on $\mathcal{P}$, then $M$ is a direct product of some isomorphic simple groups. Let $T$ denote a non-abelian simple subgroup of $M$. The O'Nan Scott Theorem divides primitive group $G$ into five types:
	\begin{enumerate}
		\item[(i)] Affine type, $M=Z^{m}_{p}\leq G\leq\mathrm{AGL}(m,p)$;% and $Z^{m}_{p}$ acts regularly on $\mathcal{P}$ ;
		\item[(ii)] Almost simple type, $M=T\leq G\leq\mathrm{Aut}(T)$; %and $T$ is the unique minimal normal subgroup of $G$;
		\item[(iii)] Twisted wreath type, $M=T^{m}\leq G\leq T^{m}:\mathrm{S}_{m}$, where   $T^{m}$ acts regularly on $\mathcal{P}$ and $m\geq6$; 
		\item[(iv)] Simple diagonal type, $M=T^{m}\leq G\leq T^{m}:(\mathrm{Out}(T)\times \mathrm{S}_{m})$, where $m\geq2$ and $|\mathcal{P}|=|T|^{m-1}$;
		\item[(v)] Product action type, $M=T^{m}\leq G\leq H\wr S_{m}$ , where $H$ has a primitive action (of almost simple or simple diagonal type) on a set $\Gamma$ of size $\omega\geq5$ and $m\geq2$.
	\end{enumerate}
	
	In order to prove  Theorem \ref{main}, it suffices to show that  (iii)-(v) cannot occur. We assume the following hypothesis.
	
 \begin{hypothesis}\label{hy1}
			Let $\mathcal{D}$ be a non-trivial quasi-symmetric $2$-design with block intersection numbers $x=0$ and $2\leq y\leq 10$ admitting a flag-transitive point-primitive group $G$ of automorphisms. 
		\end{hypothesis}
 
	Before proving the main theorems, we  introduce the following Computational method.
	
	\subsection{Computational method}\label{Comp}
	
	Given the value of $v$, we will use Magma \cite{Bosma} to search for positive integers $(v,b,r,k,\lambda)$ that satisfy the following conditions  as potential parameters for the design $\mathcal{D}$.
	\begin{equation}\label{y}
		2\leq y\leq10,
	\end{equation}
	\begin{equation}\label{v}
		k<v<\frac{k^{2}-k}{y-1},
	\end{equation}
	\begin{equation}\label{r1}
		r(k-1)=\lambda(v-1),
	\end{equation}
	\begin{equation}\label{r2}
		(y-1)(r-1)=(k-1)(\lambda-1),
	\end{equation}
	\begin{equation}\label{divisible}
		y\mid k,~y\mid (r-\lambda),
	\end{equation}
	\begin{equation}\label{b}
		b=\frac{vr}{k}.\\
	\end{equation}
	We provide specific calculation steps:
	\begin{enumerate}
		\item [(1)] For a given $v$ and $y$ obtained from \eqref{y}, substitute $v$ and  $y$ into \eqref{v} to  obtain the range of positive integer $k$;
		\item [(2)] Substitute the above possible values of $k$ and $v$ into \eqref{r1} and \eqref{r2} to obtain all possible positive integer solutions for $(r,\lambda)$, where $r>k>\lambda>1$;
		\item [(3)] Check whether $(y,k,r,\lambda)$ satisfies the conditions in \eqref{divisible};
		\item [(4)] Substitute the positive integer $r$ and corresponding $v,k $ into \eqref{b} to check whether $b$ is a positive integer.
	\end{enumerate}
	Finally, we identify the parameters $(v,b,r,k,\lambda) $  that satisfy all above conditions  as potential  parameters for the design $\mathcal{D}$.
	
	\subsection{Twisted wreath type}
	
	\begin{prop}\label{tw}
		Assume  Hypothesis \ref{hy1}.  Then $G$ is not of twisted wreath type. 
	\end{prop}
	
	\begin{proof}
		Suppose that $G$ is of twisted wreath type with $\mathrm{Soc}(G)=T^{m}$ for some non-abelian simple group $T$. Let $M=T_{1}\times T_{2}\times\cdot\cdot\cdot\times T_{m}$, where each $T_{i}$ $\left(i\in\{1,2,\dots,m\}\right)$ is a non-abelian simple group and $T_{i}\cong T$. Then $M\cong T^{m}$ is regular on the points of $\mathcal{D}$ and thus  $v=|T|^{m}$. Let $\a$ be a point of $\mathcal{D}$. Then, by the proof of  \cite[Proposition 3.2]{Zhang2023}, we know that   $\bigcup^{m}_{i=1}\Gg_{i}\setminus\{\a\}$  is the union of some non-trivial $G_{\a}$-orbits and  $|\bigcup^{m}_{i=1}\Gg_{i}\setminus\{\a\}| =m\left(|T|-1\right)$, where  $\Gg_{i}=\a^{T_{i}} $ for $i\in \{1,2,\dots,m\}$. By Lemma \ref{s23}(iii),  $r\mid\lambda m\left(|T|-1\right)$ and so $\frac{r}{(r,\lambda)}\leq m\left(|T|-1\right) $. By Lemma \ref{s22}(iv), we have 
		\begin{equation}
			|T|^{m}=v\leq2(y-1)\cdot\frac{r^{2}}{(r,\lambda)^{2}}\leq2(y-1)m^{2}\left(|T|-1\right)^{2},
			\nonumber
		\end{equation}
		i.e.,  $y\geq\frac{|T|^{m}}{2m^2\left(|T|-1\right)^{2}}+1$.
		Note that  $m\geq6$ and $|T|\geq60$  since $T$ is non-abelian simple. Then $\frac{|T|^{m}}{2m^2\left(|T|-1\right)^{2}}+1\geq\frac{60^{6}}{72\cdot59^{2}}+1>186154$, contradicting the condition  $y\leq10$. This completes the proof.
		%  i.e.,  $\frac{|T|^{m}}{2m^2(|T|-1)^{2}}\leq y-1$. 
	\end{proof}
	\subsection{Simple diagonal type}
	\begin{prop}\label{sd}
		Assume Hypothesis \ref{hy1}.  Then $G$ is not of simple diagonal type. 
	\end{prop}
	
	\begin{proof}
		Suppose that $G$ is  of simple diagonal type with $\mathrm{Soc}(G)=T^{m}$ for some non-abelian simple group $T$. Similar to the case of twisted wreath type (see the proof of \cite[Proposition 3.3]{Zhang2023}), we can deduce that 
		\begin{equation}\label{SD_1}
			|T|^{m-1}=v\leq2(y-1)\cdot\frac{r^{2}}{(r,\lambda)^{2}}\leq2(y-1)m^{2}\left(|T|-1\right)^{2}.
		\end{equation}
		Furthermore, since 
		$ G_{\a}\leq\mathrm{Aut}(T)\times \mathrm{S}_{m} $ and $r$ divides $|G_{\a}|$, we obtain $r$ divides $|T||\mathrm{Out}(T)|m!$.  By Lemma \ref{s21}(i), we have $ \frac{r}{(r,\lambda)}\cdot(k-1)=\frac{\lambda}{(r,\lambda)}\cdot\left(|T|^{m-1}-1\right)$. It follows that $\frac{r}{(r,\lambda)}$ divides $|T|^{m-1}-1$, which implies that $\left(\frac{r}{(r,\lambda)},|T|\right)=1$  and hence $\frac{r}{(r,\lambda)}$ divides $|\mathrm{Out}(T)|m!$. Then, by Lemma \ref{s22}(iv), we have 
		\begin{equation}\label{SD_2}
			|T|^{m-1}=v\leq2(y-1)\cdot\frac{r^{2}}{(r,\lambda)^{2}}\leq2(y-1)|\mathrm{Out}(T)|^{2}\left(m!\right)^{2}.
		\end{equation}
		Since $|T|\geq60$ and $y\leq10$, we can deduce that $m\leq4$ from \eqref{SD_1}. 	Furthermore, from \eqref{SD_2}, we obtain $|T|\leq72|\mathrm{Out}(T)|^{2}$ if $m=2$;  $|T|^2\leq648|\mathrm{Out}(T)|^{2}$ if $m=3$;  $|T|^3\leq10368|\mathrm{Out}(T)|^{2}$ if $m=4$. By \cite[Lemma 2.3]{Tian2013}, a non-abelian simple group $T$ satisfying the inequality $|T|\leq 400|\mathrm{Out}(T)|^{2}$ is isomorphic to one of the following groups:
		\[\mathrm{PSL}(2,q) \;\;\text{for}\;\; q=5,7,8,9,11,13,16,27, \; \text{or}\;\;\mathrm{PSL}(3,4). \]
		By calculating the orders of the above groups and the  corresponding outer automorphism groups, we can obtain that $T=\mathrm{PSL}(2,q)$ if  $m=2$, where $q\in\{5,7,8,9\}$, and $T=\mathrm{PSL}(2,5)$ if $m=3$ or $4$. Thus we deduce that $v=|\mathrm{PSL}(2,5)|^2$, $|\mathrm{PSL}(2,5)|^3$  or $\left|\mathrm{PSL}(2,q)\right|$, where $q\in\{5,7,8,9\}$.  However, for these values of $v$, computation  in Magma \cite{Bosma} shows that there are no feasible parameters for $\mathcal{D}$.
		This completes the proof.
	\end{proof}
	
	\subsection{Product action type}
	
	\begin{lemma}\label{PT}
		Assume Hypothesis \ref{hy1}. If $G$ is  of product action type such that $T^{m}\leq  G \leq H\wr\mathrm{S}_{m}$ for some non-abelian simple group $T$, where $H$ has a primitive action  on a set $\Gamma$ of size $\omega$, then $m=2$.
	\end{lemma}
	\begin{proof}
		 By  \cite[Lemma 3.1]{Liang2016},  we know that $r\mid\lambda{m}(\omega-1)$ and so  $\frac{r}{(r,\lambda)}\mid{m}(\omega-1)$. Combining this with Lemma \ref{s22}(iv), we can conclude that
		\begin{equation}
			\omega^{m}=v\leq2(y-1)\cdot\frac{r^{2}}{(r,\lambda)^{2}}\leq2(y-1)m^{2}(\omega-1)^{2}, \nonumber
		\end{equation}
		i.e., $y\geq\frac{\omega^{m}}{2m^2(\omega-1)^{2}}+1$.
		We  assume  that $m\geq3$. Note that $\omega\geq5$. Then we deduce that  $10\geq y\geq\frac{\omega^{m}}{2m^2(\omega-1)^{2}}+1\geq\frac{5^{m}}{32m^2}+1$ and so $3\leq m\leq5$. Recall that $\frac{r}{(r,\lambda)}\mid  {m}(\omega-1)$. So there exists a positive integer $a$ such that $a\cdot\frac{r}{(r,\lambda)}=m(\omega-1)$. Then we have  $\frac{a^2\omega^{m}}{2m^2(\omega-1)^{2}}\leq y-1$. For $3\leq m\leq5$, $\omega\geq5$ and $2\leq y\leq10$, we get a finite number of sets of $4$-tuples $(m,a,\omega,y)$ satisfying the above inequality. For example, if $m=3$, $a=1$ and $y=10$, then we have $\omega\leq159$. For all possible  $4$-tuples $(m,a,\omega,y)$,  combining the equation  $a\cdot\frac{r}{(r,\lambda)}=m(\omega-1)$ with the conditions in Section \ref{Comp}, computation in Magma \cite{Bosma} shows that there are no feasible parameters for $\mathcal{D}$.  This completes the proof.
	\end{proof}
	
	\begin{prop}\label{pt}
		Assume  Hypothesis \ref{hy1}. Then $G$ is not of product action type.
	\end{prop}
	\begin{proof}
		With the notation in Lemma \ref{PT}, we have $m=2$ by Lemma \ref{PT}. Since  $r\mid \lambda {m}(\omega-1)$,  there exists a positive integer $a_{1}$ such that $a_{1}r=2\lambda (\omega-1)$. Combining this  with  the equation $\lambda(\omega^{2}-1)=r(k-1)$, we obtain $k-1=a_{1}\cdot\frac{\omega+1}{2}$. On the other hand,  by Lemma \ref{s22}(i), we can deduce that
		\begin{equation}
			\frac{2(\omega-1)}{a_{1}}=\frac{r}{\lambda}< \frac{r-1}{\lambda-1}=\frac{k-1}{y-1}=\frac{a_{1}}{y-1}\cdot \frac{\omega+1}{2},
			\nonumber
		\end{equation}
		
		\begin{equation}
			\frac{2(\omega-1)}{a_{1}}=\frac{r}{\lambda}> \frac{1}{2}\cdot \frac{r-1}{\lambda-1}=\frac{1}{2}\cdot\frac{k-1}{y-1}=\frac{a_{1}}{y-1}\cdot \frac{\omega+1}{4}.
			\nonumber
		\end{equation}
		So we have 
		\begin{equation}\label{a1}
			\frac{8}{3}(y-1)\leq 4(y-1)\cdot\frac{\omega-1}{\omega+1}< a_{1}^{2}<8(y-1)\cdot\frac{\omega-1}{\omega+1}<8(y-1).
		\end{equation}
		Consider the following system of equations:
		\begin{equation}\label{v2}
			v=\omega^{2},
		\end{equation}
		\begin{equation}\label{kw1}
			k-1=a_{1}\cdot\frac{\omega+1}{2},
		\end{equation}
		\begin{equation}\label{kw2}
			a_{1}r=2\lambda(\omega-1),
		\end{equation}
		\begin{equation}\label{kw3}
			(y-1)(r-1)=(k-1)(\lambda-1).
		\end{equation}
		If $y=2$, then $a_{1}=2$ by \eqref{a1}. From \eqref{kw1}, we deduce that  $k=\omega+2$. By combining \eqref{kw2} with \eqref{kw3}, we can obtain that $r=\frac{\omega(\omega-1)}{2}$ and  $\lambda=\frac{\omega}{2}$, i.e., $(v,r,k,\lambda)=\left(\omega^{2}, \frac{\omega(\omega-1)}{2},\omega+2,\frac{\omega}{2}\right)$. By Lemma \ref{s21}(ii), $b=\frac{vr}{k}=\frac{\omega^{3}(\omega-1)}{2(\omega+2)}=\frac{1}{2}\cdot(\omega^{3}-3\omega^{2}+6\omega-12)+\frac{12}{\omega+2}$, is a positive integer. Thus $(\omega+2)\mid12$ and so $\omega=10$. Then $(y,v,b,r,k,\lambda)=(2,100,375,45,12,5)$.
		
		Similarly, if  $y=3$, then $a_{1}=3$ by \eqref{a1}. We then deduce  that $(v,r,k,\lambda)=\left(\omega^{2},\frac{(2\omega-2)(3\omega-1)}{\omega+17},\frac{3\omega+5}{2},\frac{9\omega-3}{\omega+17}\right)$. Since $\lambda=\frac{9\omega-3}{\omega+17}=9-\frac{156}{\omega+17}$ is a positive integer,  it follows that $(\omega+17)\mid156$, so $\omega=9,22,35,61$ or $139$. However, $k$ is not an integer if $\omega=22$ and  $b=\frac{vr}{k}$ is not an integer if $\omega=9,35,61$ or $139$. Hence  we exclude the case where $y=3$. By using Magma \cite{Bosma}, we do similar calculation for the cases $4\leq y\leq10$ and  obtain all potential parameters $(v,b,r,k,\lambda)$ for  $\mathcal{D}$.  These parameters and the corresponding $\mathrm {Soc}(H)$ are presented in  Table \ref{soc} by \cite[Table B.2 and  B.3]{Dixon}.
		
		\begin{table}[h]
			\centering
			\caption{Potential parameters and $\mathrm{Soc}(H)$ for Proposition \ref{pt}}\label{soc}
			\begin{tabular}{cccc}
				\toprule
				~~~~~~$y$~~~~~~   &   ~~~~~~$a_{1}$~~~~~~  & ~~~~~~$(v,b,r,k,\lambda)$~~~~~~ & ~~~~~~$\mathrm{Soc}(H)$~~~~~~ \\
				\midrule
				2 & 2 & $(10^2,375,45,12,5)$ & $\mathrm{A}_{5}$, $\mathrm{A}_{6}$, $\mathrm{A}_{10}$\\
				5 & 4 & $(21^{2},980,100,45,10)$ & $\mathrm{A}_{7}$, $\mathrm{A}_{21}$, $\mathrm{PSL}(2,7)$,  $\mathrm{PSL}(3,4)$\\
				5 & 4 & $(111^{2},165649,3025,225,55)$ & $\mathrm{A}_{111}$\\
				7 & 5 & $(77^{2},13794,456,196,15)$ &  $\mathrm{A}_{77}$, $\mathrm{M}_{22}$\\
				10 & 6 & $(92^{2},41262,1365,280,45)$ &$\mathrm{A}_{92}$\\
				\bottomrule
			\end{tabular}
		\end{table}
		
		 Let $T=\mathrm{Soc}(H)$ and $\alpha=(\delta,\delta)$ be a point. Suppose that  $T$ is $\mathrm{A}_{n}$ of degree $n$ acting on $\Gamma$ listed in Table \ref{soc}.  Then we see that $2\leq a_1<\lambda\leq \omega-2$. In fact, we have $a_1=|B\cap \Theta_1|$  by \cite[3.3.2]{Zhang2024}, where $\Theta_1=\left(\{\delta\}\times\{\Gamma\setminus\delta\}\right)\cup\left(\{\Gamma\setminus\delta\}\times\{\delta\}\right)$ and   $B$ is a block containing $\alpha$. The flag-transitivity of $G$ implies that $a_1$ is independent of the choice of $B$.  By the proof of \cite[Lemma 3.8]{Zhang2024}, we have $\omega-2\leq a_1$, which is a contradiction.
		 We now deal with the remaining cases. By  \cite[Lemma 3.10]{Tian2013}, we have 
		\begin{equation}\label{r}
			r\mid \frac{2|T|^{2}|\mathrm{Out}(T)|^{2}}{\omega^{2}}. 
		\end{equation}
		We can check that condition \eqref{r} cannot hold if $T=\mathrm{A}_{5},\mathrm{A}_{6}$, $\mathrm{PSL}(2,7)$ or $\mathrm{M}_{22}$ . Thus we only leave the cases where $T=\mathrm{A}_{7}$ or $\mathrm{PSL}(3,4)$. Note that $G$   has no conjugacy class of subgroups with index $980$ if $T=\mathrm{PSL}(3,4)$ by Magma \cite{Bosma}, a contradiction. Thus we suppose that $T=\mathrm{A}_{7}$ and  $(v,b,r,k,\lambda)=(21^{2},980,100,45,10)$. We only give  details for the case $G=\mathrm{A}_{7}\wr \mathrm{S}_{2}$ with the help of Magma \cite{Bosma} and the other cases are similar. We get $G$ by command $\mathbf{PrimitiveGroup(441,4)}$. Since $G$ is flag-transitive, $G$ acts transitively on $\mathcal{B}$. Then $|G_{B}|=|G|/b=12960$ for any $B\in\mathcal{B}$. Using command $\mathbf{Subgroups(G:OrderEqual:=12960)}$, it turns out that $G$ has two conjugacy classes of subgroups  with index $980$, denote by $H_1$, $H_2$ as representatives, respectively. As $G$ acts transitively on $\mathcal{B}$, there exists a block $B$ such that $G_{B}=H_1$ or $H_2$. Using commands $\mathbf{Orbits(H_1)}$ and $\mathbf{Orbits(H_2)}$, we see that there is only one orbit $\Delta$ with size $45$ of $H_1$,  and no orbit with  size $45$ of $H_2$  acting on $\Omega=\{1,2,...,441\}$. Thus we choose the orbit $\Delta$ as the basic block $B_1$. Using command $\mathbf{Design<2,441|B_1^G>}$,  then Magma shows that resulting structure is not a $2$-design. Thus this flag-transitive $2$-design does not exist. This completes the proof.
	\end{proof}
	\noindent\textbf{Proof of Theorem \ref{main}.} It follows immediately from Proposition \ref{tw}, \ref{sd}  and \ref{pt}.
	
	\section{Proof of Theorem \ref{second}}\label{s4}
	
	In this section, under hypothesis \ref{hy1}, we further assume that  $G$ is an almost simple group with $\mathrm{Soc}(G)=\mathrm{A}_{n}$.
	\begin{hypothesis}\label{hy_alternating}
		Let $\mathcal{D}$ be a non-trivial quasi-symmetric $2$-design with block intersection numbers $x=0$ and $2\leq y\leq10$, let $G \leq \mathrm{Aut}(\mathcal{D})$ be flag-transitive  and point-primitive with $\mathrm{Soc}(G)=\mathrm{A}_{n}$. Let $\a$ be a point of $\mathcal{D}$ and $H=G_{\a}$.
	\end{hypothesis}
	
	Since $G$ is point-primitive, $H$ is a maximal subgroup of $G$. We first assume that $n=6$. Then $G\cong \mathrm{M}_{10},\mathrm{PGL}(2,9)$ or $\mathrm{P}\Gamma\mathrm{L}(2,9)$. Each of these groups has exactly three maximal subgroups with index greater than $2$, specifically, with indices $45,36$ and $10$ respectively. 
	From $v=[G:H]$, we can obtain that $v=45,36$ or $10$. Then computation in Magma \cite{Bosma} gives  all potential parameters $(y,v,b,r,k,\lambda)$ of $\mathcal{D}$: $(5,45,66,22,15,7)$ and $(9,36,70,35,18,17)$. However, we can check that $r\nmid |H|$ for these cases, which contradicts the  flag-transitivity of $G$. Thus we have  $G=\mathrm{A}_{n}$ or $\mathrm{S}_{n}$ with $n\geq5$. Let  $\Omega_{n}=\{1,2,\dots,n\}$. By Lemma \ref{max}, the action of $H$ on $\Omega_{n}$ can be one of the following three cases: (i) primitive; (ii) transitive and imprimitive; (iii) intransitive. We analyze each of these actions separately, following similar approaches in \cite{Zhang2023DM,Zhou2015,Delandtsheer2001}.
	
	\subsection{$H$ acts primitively  on $\Omega_{n}$}

	\begin{prop}\label{primi}
		Assume  Hypothesis \ref{hy_alternating}. Then the  point-stabilizer $H=G_{\a}$ cannot act primitively on $\Omega_{n}$.
	\end{prop}
	
	\begin{proof}
		We claim that if  $G=\mathrm{S}_{n} $ then $H$ cannot be $\mathrm{A}_{n}$ or $\mathrm{S}_{n-1}$. First we have $H\neq\mathrm{A}_{n}$ since  $[G: H]=v>2$. Next suppose that  $H=\mathrm{S}_{n-1}$ so that $v=n$. By \cite[Theorem 14.2]{Wielandt1964}, we have $[G: H]=v\geq\left[\frac{n+1}{2}\right]!$ and so $v=n\leq6$. Then Magma \cite{Bosma} shows that there are no feasible parameters for $\mathcal{D}$. 
		
		Assume that $\frac{r}{(r,\lambda)}$ is even, it follows that $v$ is odd. By Lemma \ref{odd}, $v=15$ and $G=\mathrm{A}_{7}$ or $\mathrm{A}_{8}$. However, for $v=15$, computation in Magma \cite{Bosma} shows that there are no feasible parameters for $\mathcal{D}$. Thus $\frac{r}{(r,\lambda)}$ is odd. 
		By the proof of  \cite[Proposition 3.1]{Zhang2023DM} (also see \cite{Zhou2015,Delandtsheer2001}),  we have  that $\frac{r}{(r,\lambda)}$ is either a prime, namely $n-2$, $n-1$ or $n$, or the product of two primes, namely  $(n-2)n$. Moreover, since  $H$ acts primitively on $\Omega_{n}$ and $H\ngeq A_{n}$, we have $v\geq\frac{\left[\frac{n+1}{2}\right]!}{2}$ by \cite[Theorem 14.2]{Wielandt1964}. By Lemma \ref{s22}(iv), we deduce that 
		\begin{equation}\label{n_value}
			\frac{\left[\frac{n+1}{2}\right]!}{2}\leq v\leq2(y-1)\cdot\frac{r^2}{(r,\lambda)^2}\leq 18\cdot\frac{r^2}{(r,\lambda)^2}.
		\end{equation}
		On the other hand, according to Lemma \ref{s23}(i),  we have $|G_{\alpha}|^{3}>\lambda|G|>|G|$. Therefore, we only need to consider the cases in Lemma \ref{large}. We examine whether \eqref{n_value} holds for the cases in Lemma \ref{large} and determine  the feasible values of  $v$ and $\frac{r}{(r, \lambda)} $ as shown  in Table \ref{Primitive}. %Note that we have shown that $v>6$ and $v\neq15$. %
		For the values of $v$ in Table \ref{Primitive}, computation  in Magma \cite{Bosma} shows that there are no feasible parameters for $\mathcal{D}$. This completes the proof.

		\begin{table}[htb]
			\centering
			\caption{The feasible values of $v$ and $\frac{r}{(r,\lambda)}$ in Proposition \ref{primi}}\label{Primitive}
			\begin{tabular}{ccccc}
				\toprule
				~~~$v$~~~ & ~~~ $n$ ~~~& ~~~ $\frac{r}{(r,\lambda)}$~~~  &  ~~~$G$~~~  &~~~ $H$~~~\\
				\midrule
				120 & 7 &  5,7,35 &  $\mathrm{S}_{n}$ & $\mathrm{AGL}(1,7)$\\ 
				120 & 8 &  7 &  $\mathrm{S}_{n}$ & $\mathrm{PGL}(2,7)$\\     
				120 & 9 & 7 &  $\mathrm{A}_{n}$ & $\mathrm{P}\Gamma\mathrm{L}(2,8)$ \\
				840 & 9 &  7 &  $\mathrm{A}_{n}$ & $3^{2}.\mathrm{SL}(2,3)$\\ 
				840 & 9 &  7 &  $\mathrm{S}_{n}$ & $\mathrm{AGL}(2,3)$\\       
				\bottomrule
			\end{tabular}
		\end{table}
	\end{proof}
	
	\subsection{$H$ acts transitively and imprimitively  on $\Omega_{n}$}
	
	\begin{prop}
		Assume  Hypothesis  \ref{hy_alternating}. Then the  point-stabilizer  $H=G_{\a}$ cannot act transitively and imprimitively on $\Omega_{n}$.
	\end{prop}
	
	\begin{proof}
		Suppose to the contrary that $\Sigma=\{\Delta_{0},\Delta_{1},\dots,\Delta_{t-1}\}$ is a non-trivial partition of $\Omega_{n}$ preserved by $H$, where $|\Delta_{i}|=s$, $0\leq i \leq t-1$, $s,t\geq2$ and $st=n$. Then
		\begin{equation}\label{vnumber}
			v=\binom{ts-1}{s-1}\binom{(t-1)s-1}{s-1}\dots\binom{3s-1}{s-1}\binom{2s-1}{s-1}.
		\end{equation}
		Moreover, the set $O_{j}$ of $j$-cyclic partitions with respect to $X$(a partition of $\Omega_{n}$ into $t$ classes each of size $s$) is a union of orbits of $H$ on $\cP$ for $j=2,\dots,t$ (see \cite{Delandtsheer2001,Zhou2010} for definition and details).
		
		\textbf{Case(1)}: Suppose first that $s=2$. Then $t\geq3$, $v=(2t-1)(2t-3)\dots5\cdot3$, and
		\begin{equation}
			d_{j}=|O_{j}|=\frac{1}{2}\binom{t}{j}\binom{s}{1} ^{j}=2^{j-1}\binom{t}{j}.
			\nonumber
		\end{equation}
		By Lemma \ref{s23}(iii), we have $\frac{r}{(r,\lambda)}\mid d$, where $d$ is any non-trivial subdegree of $G$. Thus by Lemma \ref{s22}(iv), we have
		\begin{equation}
			v\leq2(y-1)\cdot\frac{r^{2}}{(r,\lambda)^{2}}\leq2(y-1)d^{2}. \nonumber
		\end{equation}
		If $t\geq7$, then we can check that $v=(2t-1)(2t-3)\dots5\cdot3>18t^{2}(t-1)^{2}\geq2(y-1)d_2^2\geq v$, a contradiction. Thus $t\leq6$ and so $v=15,105,945$ or $10395$.  Then Magma \cite{Bosma} shows that there are no feasible parameters for $\mathcal{D}$.

		\textbf{Case(2)}: Now we consider that $s\geq3$, then $O_{j}$ is an orbit of $H$ on $\cP$, and $d_{j}=|O_{j}|=\binom{t}{j}\binom{s}{1} ^{j}=s^{j}\binom{t}{j}$. Thus we have $\frac{r}{(r,\lambda)}\leq d_{2}=s^{2}\binom{t}{2}$ by Lemma \ref{s23}(iii). Moreover, since  $\binom{is-1}{s-1}=\frac{is-1}{s-1}\cdot\frac{is-2}{s-2}\dots \frac{is-(s-1)}{1}>i^{s-1}$ for $i=2,3,\dots,t$, we can deduce that $v>(t!)^{s-1}$. Then by Lemma \ref{s22}(iv), 
		
		\begin{equation}\label{st_bound}
			(t!)^{s-1}<v\leq2(y-1)\cdot\frac{r^2}{(r,\lambda)^2}\leq 2(y-1)d_{2}^{2}\leq9s^{4}t^{2}(t-1)^{2}/2.
		\end{equation}
		By using Magma \cite{Bosma}, we get $31$ pairs $(t,s)$ satisfying \eqref{st_bound} as  follows: $(t=2,3\leq s\leq 23),(t=3,3\leq s\leq 8),(t=4,3\leq s\leq 5)$ and $(t=5,s=3)$. Furthermore, we substitute these pairs $(t,s)$ into  \eqref{vnumber} and verify  the inequality $v\leq9s^{4}t^{2}(t-1)^{2}/2$ holds. Then we get  $11$ pairs $(t,s)$:   $(t=2, 3\leq s\leq 10),(t=3,3\leq s\leq 4),(t=4,s=3)$. Combining \eqref{vnumber} with the condition that $\frac{r}{(r,\lambda)}$ divides $d_{2}=s^{2}\binom{t}{2}$, computation  in Magma \cite{Bosma} shows that there are no feasible parameters for $\mathcal{D}$. This completes the proof. 
	\end{proof}
	
	\subsection{$H$ acts intransitively on $\Omega_{n}$}
	\begin{prop}\label{intransitive}
		Assume  Hypothesis  \ref{hy_alternating}. Then the  point-stabilizer  $H=G_{\a}$ cannot act intransitively on $\Omega_{n}$.
	\end{prop}
	
	\begin{proof}
		Suppose that $H$ acts intransitively on $\Omega_{n}$. Then  $H=(\mathrm{S}_{l}\times \mathrm{S}_{m})\cap G$ by Lemma \ref{max}(i), where $n=l+m$. Without loss of generality, we may assume that $l<m$. By the flag-transitivity of $G$, $H$ is transitive on the blocks trough $\a$ and so $H$ fixes exactly one point in $\cP$. Since $H$ stabilizes only one $l$-subset of $\Omega_{n}$, we can identify the point $\a$ with  this $l$-subset. Then $\cP$ can be identified with the set of $l$-subsets of $\Omega_{n}$. So $v=\binom{n}{l}$, $G$ has rank $l+1$ and the subdegrees are:
		\begin{equation}
			d_{0}=1, d_{i+1}=\binom{l}{i} \binom{n-l}{l-i}, i=0,1,2,\dots,l-1.
			\nonumber
		\end{equation}
		By considering the subdegree $d_l=lm$, we can deduce that 
		\begin{equation}\label{v_l}
			v=\binom{l+m}{l}\leq2(y-1)\cdot\frac{r^{2}}{(r,\lambda)^{2}}\leq2(y-1)l^{2}m^{2}\leq18l^{2}m^{2}.
		\end{equation}
		Here, we have used Lemma \ref{s23}(iii) and \ref{s22}(iv).
		If $m>l\geq10$, then \begin{align*}
			\binom{l+m}{l}=&\frac{(m+l)(m+l-1)\dots (m+1)}{l!}\\=&(1+m)\left(1+\frac{m}{2}\right)\dots \left(1+\frac{m}{l}\right)\\ \geq& (1+m)\left(1+\frac{m}{2}\right)\dots \left(1+\frac{m}{10}\right)>18m^{4}>18l^{2}m^{2},
		\end{align*} which contradicts  \eqref{v_l}. Thus $l\leq9$.
		
		\textbf{Case(1)}: If $l=1$, then $v=n\geq5$ and the subdegrees are $1,n-1$. Since $4\leq k \leq v-2$ and $G$ is $(v-2)$-transitive on $\cP$, we have $G$ acts $k$-transitively on $\cP$. Then $\binom{n}{k}=b\leq v(v-1)/k=n(n-1)/k$ by Lemma \ref{s22}(ii), which is impossible for $k\geq4$.
		
		\textbf{Case(2)}: If $l=2$, then $v=\frac{n(n-1)}{2}$ and subdegrees are $1,\binom{n-2}{2},2(n-2)$. By Lemma \ref{s23}(iii), $\frac{r}{(r,\lambda)}$ divides $\left(2(n-2),\binom{n-2}{2}\right)=\frac{n-2}{2},n-2$ or $2(n-2)$ with $n\equiv 0\pmod{2},n\equiv 1\pmod{4}$ or $n\equiv 3\pmod{4}$ respectively. Thus we have $\frac{r}{\lambda}=\frac{r}{(r,\lambda)}/\frac{\lambda}{(r,\lambda)}=\frac{2(n-2)}{u}$, where $u$ is a positive integer. Moreover, $n\equiv 3\pmod{4}$ if $u$ is odd.  Note that
		\[(y-1)\cdot\frac{r}{\lambda}<k-1=(y-1)\cdot\frac{r-1}{\lambda-1}<2(y-1)\cdot\frac{r}{\lambda}\]
		and $k-1=(v-1)\cdot\frac{\lambda}{r}$ by Lemma \ref{s21}(i) and \ref{s22}(i). Then we obtain 
		\[(y-1)\cdot\frac{r}{\lambda}<k-1=(v-1)\cdot\frac{\lambda}{r}<2(y-1)\cdot\frac{r}{\lambda},\]
		and so
		\begin{equation}\label{r_la}
			(y-1)\cdot\frac{r^2}{\lambda^2}<\frac{(n-2)(n+1)}{2}<2(y-1)\cdot\frac{r^2}{\lambda^2}.
		\end{equation}
		Suppose that $\frac{r}{\lambda}\geq n-2$. Then  $\frac{(n-2)(n+1)}{2}>(y-1)\cdot\frac{r^2}{\lambda^2}\geq
		(n-2)^2$ by \eqref{r_la}, a contradiction. Hence we have $\frac{r}{\lambda}=\frac{2(n-2)}{u}$ and so $k=\frac{(n+1)u}{4}+1$, where $u\geq3$. Furthermore, we have $u^2<4(y-1)\cdot\frac{4(n-2)}{n+1}<16(y-1)\leq 144$ by \eqref{r_la}, and so $3\leq u\leq11$. Here, we detail the case where 
		$u$ is even and the odd case follows similarly. Suppose that $u$ is even. Then $u=4,6,8$ or $10$.
		%(a) If $n\equiv0\pmod{2}$, then $\frac{r}{(r,\lambda)}\mid \frac{n-2}{2}$. Let $\frac{r}{(r,\lambda)}\cdot u=\frac{n-2}{2}$, where $u$ is a positive integer . By Lemma \ref{s23}(iii), we have $\frac{n(n-1)}{2}=v\leq2(y-1)\frac{r^{2}}{(r,\lambda)^{2}}\leq\frac{9(n-2)^2}{2u^2}$, and so  $ u^{2}\leq\frac{9(n-2)^{2}}{n(n-1)}<9$. Then we can obtain that $u=1$ or $2$.
		
		(a) If $u=4$, then  $k=n+2$  and $3\leq y \leq 4$ from \eqref{r_la}. Furthermore, if $y=4$, then $n<8$, which implies that $n=6$ since $y$ divides $k$. Then $v=15$, which is impossible by the proof of Proposition \ref{primi}. Thus we have $y=3$.   By considering the following equations:
		\begin{gather}
			vr=bk,\nonumber\\
			r(k-1)=\lambda(v-1),\nonumber\\
			(y-1)(r-1)=(k-1)(\lambda-1),\nonumber
		\end{gather}
		we can deduce that  $(b,r,\lambda)=\left(\frac{n(n-1)^2(n-2)}{12(n+2)},\frac{(n-1)(n-2)}{6},\frac{n-1}{3}\right)$. As $b$ is a positive integer, it follows that $n+2$ divides $72$. By  checking each possible value of $n$, we get the unique integer solution: $(y,n,v,b,r,k,\lambda)=(3,10,45,45,12,12,3)$. However, $b=v$ which is impossible.
		
		(b) If $u=6$, then  $k=\frac{3(n+1)}{2}+1$  and $4\leq y \leq 8$ from \eqref{r_la}. Furthermore, $n<29$ if $y=6$; $n<11$ if $y=7$; $n=5$ if $y=8$. For the cases $y\geq 6$, we easily check that there are no feasible parameters for $\mathcal{D}$. We only give   details for the case $y=4$ since the other case is similar. Suppose that $y=4$. Then we deduce that $\lambda=\frac{3n-3}{n+7}=3-\frac{24}{n+7}$ similar to the case (a). Hence $n=17$ and so $k=28,r=10,v=136$, which implies $b$ is not an integer: a contradiction.

		(c) If $u=8$, then  $k=2n+3$  and $6\leq y\leq 10$ from \eqref{r_la}.  Furthermore, we have $y=7$ or $9$ from $y\mid k$.  If $y=7$, then we can deduce that $\lambda=4+\frac{32}{n-10}$ and so $n=12,14,18,26$ or $42$, which contradicts  that $y=7$ divides $k=2n+3$. If $y=9$, then we deduce that  $(b,r,\lambda)=\left(\frac{n(n-1)(n-2)(n-3)}{24(2n+3)},\frac{(n-2)(n-3)}{12},\frac{n-3}{3}\right)$. As $b$ is a positive integer, it follows that $2n+3$ divides $945$. By checking each possible value of $n$, we get the unique integer solution: $(y,n,v,b,r,k,\lambda)=(9,30,435,435,63,63,9)$. However, $b=v$ which is impossible.

		(d) If $u=10$, then $k=\frac{5(n+1)}{2}+1$  and $8\leq y\leq 10$ from \eqref{r_la}.  We only give   details for the case $y=8$ since  other cases are similar.   If $y=8$, then we can deduce that $\lambda=\frac{5(5n-9)}{11n+53}<3<y$, which is impossible.
		
		\textbf{Case(3)}: Suppose that $3\leq l\leq9$. Recall that $m$ needs to satisfy the inequality $\binom{l+m}{l}\leq2(y-1)l^{2}m^{2}$ by \eqref{v_l}. Therefore, given specific values of $y$ and $l$, we can determine the range of $m$ and subsequently the range of 
		$v$. Furthermore, when $v$ takes a fixed value, we can calculate the Greatest Common Divisor of two non-trivial subdegrees  $a:=(d_1,d_l)$, where  $d_1=\binom{m}{l}$ and $d_l=lm$. Then we verify whether the inequality $v\leq2(y-1)\cdot\frac{r^2}{(r,\lambda)^2}\leq2(y-1)a^2$ holds. Finally, by considering the  Computational method described  in Section \ref{Comp}, we can obtain the unique potential parameters: $(y,n,v,b,r,k,\lambda)=(3,8,56,210,45,12,9)$ by Magma \cite{Bosma}. By \cite[Theorem 7]{Munemasa2020}, a quasi-symmetric $2$‐$(56, 12, 9)$ design with block intersection numbers $0$ and $3$ does not exist. In fact, according  to  \cite[Table 2]{Tang2022}, there does not exist flag-transitive point-primitive non-symmetric $2$-$(56,12,9)$ design with an almost simple automorphism group. Thus we rule out this case. This completes the proof. 
	\end{proof}
	
	\noindent\textbf{Proof of Theorem \ref{second}.} It follows immediately from Proposition \ref{primi}-\ref{intransitive}.
	
	\section{Proof of Theorem \ref{third}}\label{s5}
	
	Let  $T$  be an arbitrary  sporadic simple group.  Then  $\mathrm{Aut}(T)=T$ or $\mathrm{Aut}(T)=T:2$ and so $G=T$ or $T:2$ by \cite{Conway1985}. The lists of maximal subgroups of $T$ and $\mathrm{Aut}(T)$ are given   in \cite {Conway1985} (also see \cite{Wilson2017}) except for the Monster group $\mathrm{M}$. The possible  maximal subgroups $H$ of $\mathrm{M}$, which are
	not listed in \cite{Conway1985} are given in \cite{Bray2006}. They are of the form $N\unlhd H\leq\mathrm{Aut}(N)$, where $N$ is isomorphic to one of the following simple groups: $\mathrm{PSL}(2,13),\mathrm{PSU}(3,4),\mathrm{PSU}(3,8),\mathrm{Sz}(8)$.
	It is easily shown that $|H|^3\leq|\mathrm{Aut}(N)|^3<|M|$ for each case, which
	contradicts Lemma \ref{s23}(i). Hence, $N$
	can be ruled out.
	
	For remaining almost simple groups $G$ and their maximal subgroups $G_\alpha$. Let $r_ {max}=(v-1, G_{\alpha})$. Then from Lemma \ref{s23}(ii), we have  $\frac{r}{(r,\lambda)}\mid r_{max}$. Combining this  with Lemma \ref{s22}(iv), we obtain  that $v \leq2(y-1)\cdot\frac{r^{2}}{(r,\lambda)^{2}}\leq 18r_{max}^{2} $. Thus we first obtain pairs $(G,G_\alpha)$  that satisfy  $|G_ {\alpha}|^ {3}>\lambda |G|\geq2| G |$ and $[G:G_{\alpha}]=v\leq 18r_{max}^{2}$. Then, for these pairs, we compute all possible parameters $(y,v,b,r,k,\lambda)$ that satisfy the conditions \eqref{y}-\eqref{b} given in Section \ref{Comp}. With the help of Magma \cite{Bosma}, we obtain four potential parameters $(y,v,b,r,k,\lambda)$ as shown in Table \ref{sporadic}.
	
	\begin{table}[htb]
		\centering
		\caption{Potential parameters for $\mathcal {D} $}\label{sporadic}
		\begin{tabular}{ccccc}
			\toprule
			Case  &   $(y,v,b,r,k,\lambda)$  &  $\frac{r}{(r,\lambda)}$ &  $r_{max}$  &   $(G,G_\alpha)$  \\
			\midrule
			$(1)$ & $(3,12,22,11,6,5)$ & $11$ & $11$ & $(\mathrm{M}_{11},\mathrm{PSL}(2,11)),(\mathrm{M}_{12},\mathrm{M}_{11})$\\
			$(2)$ & $(2,22,77,21,6,5)$ & $21$ & $21$ & $(\mathrm{M}_{22},\mathrm{PSL}(3,4)),(\mathrm{M}_{22}:2,\mathrm{PSL}(3,4):2_2)$\\
			$(3)$ & $(2,24,184,46,6,10)$ & $23$ & $23$ & $(\mathrm{M}_{24},\mathrm{M}_{23})$\\
			$(4)$ & $(6,24,46,23,12,11)$ & $23$ & $23$ & $(\mathrm{M}_{24},\mathrm{M}_{23})$\\
			\bottomrule
		\end{tabular}
	\end{table}

	For case $(3)$, we see that  $y=2$ and $\lambda \nmid r$.  Then it can be concluded from \cite[Theorem 2]{Zhang2023} that this case does not exist. For cases $(1),(2)$ and $(4)$, we see that  $(r,\lambda)=1$. Then by \cite[Theorem 1]{Zhan2016},  cases $(1)$ and $(2)$  exist, but case $(4)$ does not, i.e., $\mathcal{D}$ is the unique $2$-$(12,6,5)$ design with block intersection numbers $0,3$ and $G=\mathrm{M}_{11}$, or $\mathcal{D}$ is the unique $2$-$(22,6,5)$ design with block intersection numbers $0,2$ and  $G= \mathrm{M}_{22}$ or $\mathrm{M}_{22}:2$. For  specific constructions of $\mathcal{D}$,  see \cite[3.3.1 and 3.3.2]{Zhan2016}. This completes the proof of Theorem \ref{third}. 
	
	\section{Declaration of competing interest}
	
	The authors declare that they have no known competing financial interests or personal relationships that could have appeared to influence the work reported in this paper.
	
	\section{Acknowledgments}
	We would like to thank the anonymous referees whose helpful comments have improved
	this paper.
	\vspace*{10pt}

	\begin{center}
		\scriptsize
		\setlength{\bibsep}{0.5ex}  % vertical spacing between references
		\linespread{0.5}
		%\bibliography{PGL}
		\bibliographystyle{plain}

	\end{center}
	
\end{document}